\documentclass[12pt,epsfig,amsfonts]{amsart} 
\usepackage{amsmath,amsthm,amssymb,amscd,epsfig}
\usepackage{graphicx}

\newtheorem*{theorema}{Theorem A}
\newtheorem*{theoremb}{Theorem B}
\newtheorem*{theoremc}{Theorem C}
\newtheorem*{theoremd}{Theorem D}

\newtheorem{lemma}{Lemma}[section]
\newtheorem{sublemma}[lemma]{Sublemma}
\newtheorem{remark}[lemma]{Remark}
\newtheorem{prop}[lemma]{Proposition}
\newtheorem{cor}[lemma]{Corollary}
\newtheorem{definition}[lemma]{Definition}
\newtheorem{claim}[lemma]{Claim}

\begin{document}
\author{Hiroki Takahasi}

\address{Department of Mathematics,
Keio University, Yokohama,
223-8522, JAPAN} 
\email{hiroki@math.keio.ac.jp}
\subjclass[2010]{37D25, 37D35, 37E30, 37G25}

\title[Equilibrium measures at temperature zero for H\'enon-like maps] 
{Equilibrium measures at temperature zero\\ for H\'enon-like maps 
at the first bifurcation}

\begin{abstract}
We develop a thermodynamic formalism for a strongly dissipative H\'enon-like map at the first bifurcation parameter
at which the uniform hyperbolicity is destroyed by the formation of tangencies inside the limit set.
For any $t\in\mathbb R$
we prove the existence of an invariant Borel probability measure which minimizes
the free energy associated with a non continuous geometric potential $-t\log J^u$, where $J^u$ denotes the Jacobian in the unstable direction.
Under a mild condition, we show that any accumulation point of these measures as $t\to+\infty$ minimizes the unstable Lyapunov exponent.
We also show that
 the equilibrium measures converge as $t\to-\infty$ to a Dirac measure which maximizes the unstable Lyapunov exponent.
\end{abstract}

\maketitle

\section{introduction}
A basic problem in dynamics is to describe how structurally stable 
systems lose their stability
through continuous modifications of the systems.
The loss of stability of horseshoes through 
homoclinic bifurcations is modeled by 
a family of H\'enon-like diffeomorphisms
\begin{equation}\label{henon}
f_a\colon(x,y)\in\mathbb R^2\mapsto(1-ax^2,0)+b\cdot\Phi(a,b,x,y),\quad a\in\mathbb R, \ 0<b\ll1.\end{equation}
Here, $\Phi$ is bounded continuous in $(a,b,x,y)$ and 
$C^2$ in $(a,x,y)$. 
It is known \cite{BedSmi06,CLR08,DevNit79,Tak13} that 
there is a  \emph{first bifurcation parameter}
 $a^*=a^*(b)\in\mathbb R$ 
with the following properties:


\begin{itemize}

\item $a^*\to2$ as $b\to0$;

\item the non wandering
set of $f_a$ is a uniformly hyperbolic horseshoe for $a>a^*$ ;

\item for $a=a^*$ there is a single orbit of homoclinic or heteroclinic tangency involving (one of) the two fixed saddles.
The tangency is quadratic, and the family $\{f_a\}_{a\in\mathbb R}$ unfolds this tangency generically.

\end{itemize}

The study of the map $f_{a^*}$ opens the door to understanding the dynamics
beyond uniform hyperbolicity in dimension two.
In this paper we advance
the thermodynamic formalism for $f_{a^*}$ initiated in \cite{SenTak1,SenTak2}.
We prove the existence of equilibrium measures for a family $\{\varphi_t\}_{t\in\mathbb R}$ of non continuous geometric potentials, and study 
accumulation points of these measures as $t\to\pm\infty$.

Write $f$ for $f_{a^*}$.
 The non wandering set of $f$, denoted by $\Omega$, is a compact $f$-invariant set. 
Let 
$\mathcal M(f)$ denote the space of $f$-invariant
Borel probability measures endowed with the topology of weak convergence.
For a potential function $\varphi\colon\Omega\to\mathbb R$ the minus of the
free energy $F_{\varphi}\colon\mathcal M(f)\to\mathbb R$ is defined by 
$$F_\varphi(\mu)=h(\mu)+\int\varphi d\mu,$$
where $h(\mu)$ denotes the entropy of $\mu$.
An \emph{equilibrium measure} for the potential $\varphi$ is a measure $\mu_\varphi\in\mathcal M(f)$ which maximizes
$F_\varphi$, i.e.,
$$F_\varphi(\mu_\varphi)=\sup\{F_\varphi(\mu)\colon\mu\in\mathcal M(f)\}.$$
The existence and uniqueness of equilibrium measures depend upon the characteristics of the system and the potential.
The family of potentials we are concerned with is 
$$\varphi_t=-t\log J^u\quad t\in\mathbb R,$$ where $J^u$ denotes the Jacobian in the \emph{unstable direction} defined as follows. 
For a point $x\in \mathbb R^2$ let $E_x^u$ denote the one-dimensional subspace of $T_x\mathbb R^2$ such that
\begin{equation}\label{eu}
\limsup_{n\to\infty}\frac{1}{n}\log\|D_xf^{-n}|E_x^u\|<0.\end{equation}
Since $f^{-1}$ expands area, the one-dimensional subspace of $T_x\mathbb R^2$ with this property is unique 
when it makes sense. 
We call $E^u_x$ the {\it unstable direction at $x$} and define $J^u(x)=\Vert D_xf|E^u_x\Vert$.
It was proved in  \cite[Proposition 4.1]{SenTak1} that $E_x^u$ makes sense for all $x\in\Omega$,
and $x\in\Omega\mapsto E_x^u$ is continuous except at the fixed saddle near $(-1,0)$
where it is merely measurable.

Since the chaotic behavior of $f$ is created by the (non-uniform) 
expansion along the unstable direction,
a good deal of information is obtained by studying the equilibrium measures for $\varphi_t$
and the associated \emph{pressure function}
$t\in\mathbb R\mapsto P(t)$, where
$$P(t)=\sup\{F_{\varphi_t}(\mu)\colon
\mu\in\mathcal M(f)\}.$$
The existence of equilibrium measures for $\varphi_t$ 
was proved in \cite{SenTak1} for all $t\leq0$, and for those $t>0$ such that
$P(t)/t$ is slightly bigger than $-\log2$. 
However, the arguments and the result in \cite{SenTak1} do not cover sufficiently large $t>0$.
Our first theorem complements this point.

\begin{theorema}
Assume $f$ preserves orientation. For any $t\in\mathbb R$ there exists an equilibrium measure for $\varphi_t$.
\end{theorema}

For $t$ in a large bounded interval,
the uniqueness of equilibrium measures for $\varphi_t$ was established in \cite{SenTak2}.
It would be nice to prove the uniqueness for all $t\in\mathbb R$, including the orientation reversing case.

Since $t$ represents the inverse of the temperature in statistical mechanics,
$t\to\pm\infty$ means that the temperature goes to zero.
Hence, it is natural to study accumulation points of equilibrium measures for $\varphi_t$ as $t\to\pm\infty$.
They represent the lowest  energy states, and may reflect the characteristics of the system.

The study of the behavior of the equilibrium measures as $t\to\pm\infty$ is also related to the ergodic optimization (See e.g. \cite{BarLepLop} and the references therein):
given a continuous dynamical system $T$ acting on a compact metric space $X$, and 
a real-valued function $\phi$ on $X$, one looks for $T$-invariant Borel probability measures
which maximize the integral of $\phi$. 
One way to do this is by freezing the system: to consider a family $\{t\phi\}_{t\in\mathbb R}$ of potentials 
and an associated family $\{\nu_t\}_{t\in\mathbb R}$ of equilibrium measures, and to let $t\to+\infty$.
If the topological entropy is finite and the potential is continuous, 
then any accumulation point as $t\to+\infty$ maximizes the integral of $\phi$.
For uniformly hyperbolic systems or the subshift of finite type,
the convergence has been established 
for certain locally constant potentials \cite{Bre03,Lep05} as well as
for a residual set of continuous potentials \cite{CLT01,JenMor08}.
However, little is known for non hyperbolic systems.




An  {\it unstable Lyapunov exponent} of a measure $\mu\in\mathcal M(f)$ is a number $\lambda^u(\mu)$ defined by 
$$\lambda^u(\mu)=\int\log J^ud\mu.$$
Of interest to us are measures which optimize the unstable Lyapunov exponent.
Since the unstable Lyapunov exponent is not continuous as a function of measures, 
the existence of such measures is an issue.
We show that  any accumulation point 
of the equilibrium measures for $\varphi_t=-t\log J^u$ as $t\to\pm\infty$
optimizes the unstable Lyapunov exponent.

Set
$$\lambda_m^u=\inf\{\lambda^u(\mu)\colon \mu\in\mathcal M(f)\}.$$
A measure $\mu\in\mathcal M(f)$ 
is called {\it Lyapunov minimizing} if $\lambda^u(\mu)=\lambda_m^u$. 
Let $Q$ denote the fixed point of $f$ near $(-1,0)$, and
$\delta_Q$ the Dirac measure at $Q$.

\begin{theoremb}
Assume $f$ preserves orientation. For $t\in\mathbb R$ let $\mu_t$ be 
an ergodic equilibrium measure for $\varphi_{t}$.
Any accumulation point of  $\{\mu_t\}_{t\in\mathbb R}$ as $t\to+\infty$
is $\delta_Q$, or a Lyapunov minimizing measure.
If $(1/2)\lambda^u(\delta_Q)\neq\lambda_m^u$, then any accumulation point of $\{\mu_t\}_{t\in\mathbb R}$ as $t\to+\infty$
is a Lyapunov minimizing measure.
\end{theoremb}

Since $\lambda^u(\delta_Q)\to\log4$ and $\lambda^u_m\to\log 2$ as $b\to0$, 
it is not easy to verify $(1/2)\lambda^u(\delta_Q)\neq\lambda^u_m$.
However, from a given family \eqref{henon} of H\'enon-like diffeomorphisms
one can construct another satisfying this condition
by slightly perturbing $\Phi$.

It is worthwhile to compare Theorem B with the results of Leplaideur \cite{Lep11}.
In this paper, he studied an orientation preserving non-uniformly hyperbolic horseshoe map with three symbols, 
with a single orbit of homoclinic tangency, introduced in \cite{Rio01}. 
Although this map is similar to our $f$ at a first glance, its equilibrium measures converge as $t\to+\infty$ to a Dirac measure
which maximizes the unstable Lyapunov exponent. He also proved the nonexistence of a measure which minimizes the unstable Lyapunov exponent.

Since there may exist multiple Lyapunov minimizing measures of $f$,
it is important to give a criterion for which one is ``selected" in the limit $t\to+\infty$.
The next theorem establishes a version of the ``entropy criterion" in \cite{BarLepLop}
for uniformly hyperbolic systems or the subshift of finite type
with H\"older continuous potentials.
Let us say that a Lyapunov minimizing measure $\mu\in\mathcal M(f)$ is {\it entropy maximizing}
if 
$$h(\mu)=\sup\{h(\nu)\colon\nu\in\mathcal M(f),\text{$\nu$ is Lyapunov minimizing}\}.$$

\begin{theoremc}
Let $f$ and $\{\mu_t\}_{t\in\mathbb R}$
be the same as in Theorem B.
If $(1/2)\lambda^u(\delta_Q)\neq\lambda_m^u$, then any accumulation point of $\{\mu_t\}_{t\in\mathbb R}$ as $t\to+\infty$
is an entropy maximizing measure.
\end{theoremc}

We now turn to the case $t\to-\infty$. 
The next theorem holds regardless of the orientation of the map $f$.

\begin{theoremd}
 Let $\{\mu_t\}_{t\in\mathbb R}$
be such that $\mu_{t}$
is an ergodic equilibrium measure for $\varphi_{t}$ for all $t\in\mathbb R$.
Then $\mu_t$ converges to $\delta_Q$ as $t\to-\infty$.
\end{theoremd}

It follows from a proof of Theorem D that $\delta_Q$ is the unique measure which maximizes
the unstable Lyapunov exponent (See Lemma \ref{maximize}).
Apart from the uniqueness, the existence of such maximizing measures follows from the result in \cite{CLR04}.

The rest of this paper consists of two sections.
In Sect.2 we develop necessary tools, and 
prove the theorems in Sect.3.
A main ingredient is a control of derivatives in the unstable direction.
To recover from small derivatives near the point of tangency, we develop Benedicks $\&$ Carleson's critical point approach
\cite{BenCar91} further.
The difference from the attractor case \cite{BenCar91} 
is that all but one critical points escape to infinity
under forward iteration. 
This issue has been successfully tackled 
in \cite{SenTak1}, but
substantial improvements are necessary to treat all $t>0$.
In particular, both lower and upper estimates of derivatives are necessary, as stated in
Proposition \ref{binding}.

In Sect.2.5 we prove a key upper estimate of $\lambda_m^u$ (See Corollary \ref{estcor})
needed for the proofs of Theorems A and B.
Since each critical orbit spends most of its lifespan near the fixed saddle
with a large derivative, the construction of measures with small unstable Lyapunov exponent involves
a control of the position at which reference orbits are released from the effect of the critical orbits.
We show that this is feasible for carefully chosen orbits,
provided the map $f$ preserves orientation.


\section{Preliminaries}
In this section we develop necessary tools for the proofs of the theorems.
For the rest of this paper 
we are concerned with the following constants: $\delta$, $b$ chosen in this order,
the purposes of which are as follows:
\begin{itemize}

\item $\delta$ determines the size of a neighborhood of $\zeta_0$ (See Sect.\ref{critical});

\item $b$ determines the magnitude of the reminder term $b\cdot\Phi$ in \eqref{henon}.
\end{itemize}
We shall write $C$ with or without indices to denote any constant which is independent of $\delta$, $b$.

\begin{figure}
\begin{center}
\includegraphics[height=6.5cm,width=14cm]
{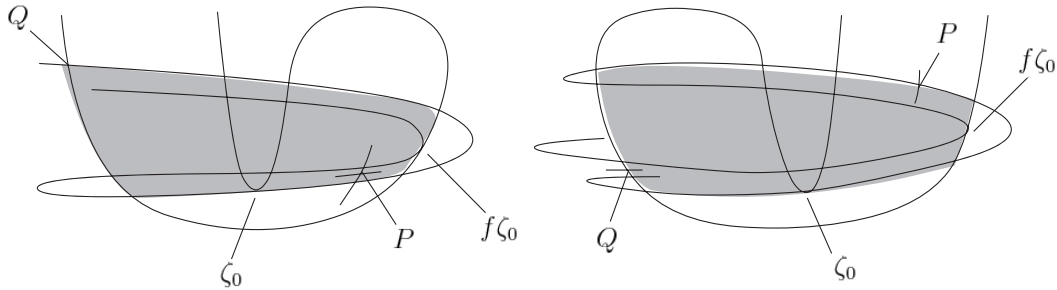}
\caption{Manifold organization for $a=a^*$ in the case of the fold turning down:
orientation preserving/reversing (left/right).
The shaded domains represent the rectangle $R$ 
containing the non wandering set $\Omega$ (see Sect.\ref{family}).}
\end{center}
\end{figure}

\subsection{The non wandering set}\label{family}
The map $f$ has exactly two fixed points, which are saddles.
Let $P$ denote the one near $(1/2,0)$. Recall that $Q$ is the other one near $(-1,0)$.
The orbit of tangency intersects a small neighborhood of the origin $(0,0)$ exactly at one point, denoted by $\zeta_0$ (FIGURE 1).
If $f$ preserves orientation, then $\zeta_0\in W^s(Q)\cap W^u(Q)$.
If $f$ reverses orientation, then $\zeta_0\in W^s(Q)\cap W^u(P)$.

If $f$ preserves orientation, let $W^u=W^u(Q)$. Otherwise,
let $W^u=W^u(P)$.
By a {\it rectangle} we mean any
compact domain bordered by two compact curves in $W^u$ and two in the
stable manifolds of $P$ or $Q$. By an {\it unstable side} of a
rectangle we mean any of the two boundary curves in $W^u$. A {\it
stable side} is defined similarly.

We define a rectangle containing the non wandering set.
Let $$V=\{(x,y)\in \mathbb R^2\colon |x|<2, |y|< \sqrt{b}\}.$$
By the results of \cite{SenTak1} there exists a rectangle $R$ in $V$ with the following properties (See FIGURE 1):

\begin{itemize}
\item[(R1)] $\displaystyle{\Omega=\{x\in R\colon f^nx\in R\ \text{ for every }n\in\mathbb Z\}}$;

\item[(R2)] one of the unstable sides of $R$ contains $\zeta_0$;

\item[(R3)] one of the stable sides of $R$ contains $f\zeta_0$. This side is denoted by $\alpha_0^+$. The other side, denoted by $\alpha_0^-$,  contains $Q$;

\item[(R4)] $f\alpha_0^+\subset\alpha_0^-$.
\end{itemize}

\subsection{Critical points}\label{critical}
Set
$$I(\delta)=\{(x,y)\in V\colon |x|<\delta\}.$$
Observe that $\zeta_0\in I(\delta)$, provided $b$ is small enough.
Although the dynamics  outside of $I(\delta)$ is uniformly hyperbolic, 
returns to the inside of $I(\delta)$ are inevitable and must be treated with care.
A key ingredient is the notion of critical points, i.e.,
points of tangencies between $C^2(b)$-curves in $W^u$ and preimages of leaves
of a stable foliation. 
We quote results from \cite{SenTak1} surrounding critical points, and shapen them further.

From
the hyperbolicity of the saddle $Q$,
there exist two mutually disjoint connected open sets $U^-$, $U^+$ independent of $b$ such that
$\alpha_0^-\subset U^-$, $\alpha_0^+\subset U^+$, $U^+\cap fU^+=\emptyset=U^+\cap fU^-$ and 
a foliation $\mathcal F^s$ of $U=U^-\cup U^+$ by one-dimensional leaves such
that: 
\begin{itemize}
\item[(F1)] $\mathcal F^s(Q)$, the leaf of $\mathcal F^s$ containing $Q$,
contains $\alpha_0^-$; 
\item[(F2)] if $x,fx\in U$, then $f(\mathcal F^s(x))
\subset\mathcal F^s(fx)$;

\item[(F3)] let $e^s(x)$ denote the unit vector in $T_x\mathcal F^s(x)$ whose second component is positive. 
Then $x\mapsto e^s(x)$ is $C^{1}$, $\|D_xfe^s(x)\|\leq Cb$ and $\|D_xe^s(x)\|\leq C$;

\item[(F4)] if $x,fx\in U$, then 
$s(e^s(x))\geq
C/\sqrt{b}$.
\end{itemize}
Here, a {\it slope} $s(v)$ of a nonzero tangent vector $v=\left(\begin{smallmatrix}\xi\\\eta\end{smallmatrix}\right)$ at a point in $\mathbb R^2$ is defined by 
$s(v)=|\eta|/|\xi|$ if $\xi\neq0$, and
$s(v)=\infty$ if $\xi=0$.

\begin{definition}
{\rm We say $\zeta\in W^u\cap I(\delta)$ is a {\it critical point} if $f\zeta\in U^+$ and
$T_{f\zeta}W^u=T_{f\zeta}\mathcal F^s(f\zeta)$.}
\end{definition}

By a \emph{$C^2(b)$-curve} we mean a compact, nearly horizontal $C^2$ curve in $V$ such that the slopes of tangent vectors to it 
are $\leq\sqrt{b}$ and the curvature is everywhere $\leq\sqrt{b}$. 
Let $S$ denote the compact lenticular domain bounded by the parabola
$f^{-1}\alpha_0^+\cap R$ and the unstable side of $R$ not containing $\zeta_0$.
Let us record two properties of the critical points:

\begin{itemize}

\item[(C1)]
any $C^2(b)$-curve in $W^u\cap I(\delta)$ contains at most one
critical point (See e.g. \cite[Remark 2.4]{Tak12});

\item[(C2)] 
any critical point other than $\zeta_0$ is contained in the interior of $S$.
Hence it is mapped to the outside of $R$, and then escape to infinity
under forward iteration.
\end{itemize}

(C2) implies that the critical orbits are contained in a region where 
the uniform hyperbolicity is apparent. Hence, by binding generic orbits which fall inside $I(\delta)$
to suitable critical points, and then copying the exponential growth along the critical orbits,
one shows that the horizontal slopes and the expansion are restored
after suffering from the loss due to the folding behavior near $I(\delta)$.

 \subsection{Binding argument}

Let $\zeta$ be a critical point and $x\in I(\delta)\setminus S$.
We say a unit tangent vector $v$ at $x$
is {\it in admissible position relative to} $\zeta$ if there exists a $C^2(b)$-curve 
which is tangent to both $T_\zeta W^u$ and $v$.

 \begin{prop}\label{binding}
 Let $\zeta$ be a critical point, $x\in (\Omega\cap I(\delta))\setminus S$  and $v$ a unit tangent vector at $x$
  in admissible position relative to $\zeta$.
There exists a positive integer $p=p(\zeta,x)$ such that:
\begin{itemize}

\item[(a)] $f^i\zeta$, $f^ix\in U$ for every $1\leq i\leq p$;

\item[(b)]
$C_1e^{\frac{p}{2}\lambda^u(\delta_Q)}\leq\|D_xf^pv\|\leq C_2e^{\frac{p}{2}\lambda^u(\delta_Q)}$;

\item[(c)]  $s(D_xf^pv)\leq\sqrt{b}$;

\item[(d)]\label{away}
if $\zeta=\zeta_0$, then $0<C_3\leq |f^px-Q|\leq C_4\ll1.$
\end{itemize}
\end{prop}

\begin{proof}
Let $\tau>0$ be sufficiently small so that
$\{y\in\mathbb R^2\colon\min\{|y-z|\colon z\in\alpha_0^-\cup\alpha_0^+\}\leq \sqrt{\tau}\}
\subset U.$
For $i\geq1$ write
$w_i=D_{f\zeta}f^{i-1}(\begin{smallmatrix}1\\0\end{smallmatrix}).$ 
For $k\geq1$
define
 $$D_k=\tau\left[\sum_{i=1}^k\frac{\|w_{i}\|^2}{\|w_{i+1}\|}\right]^{-1}.$$
Write $\mathcal F^s(f\zeta)=\{(F(y),y)\colon y\in J\}$, where
 $J$ is an interval containing $[-\sqrt{b},\sqrt{b}]$.
Write $fx=(x_0,y_0)$, and
let $\gamma$ denote the segment connecting $fx$ and $(F(y_0),y_0)$.
 Set $N=\sup\{i\geq0\colon f^i\zeta\in U\}$.
 We claim 
 there exists a unique integer $p\in[1,N]$ such that
 \begin{equation}\label{keq}D_p<{\rm length}(\gamma)\leq D_{p-1}.\end{equation}
Since $\zeta\neq x$ and $D_p\to0$ as $p\to\infty$, this is obvious if $N=\infty$. In the case $N<\infty$, assume
$ |x_0-F(y_0)|\leq  D_{N}$.
Then $|f^{N+1}\zeta-f^{N+1}x|\leq C\tau$.
From the assumption $x\in\Omega$ and (R1), $f^{N+1}x\in R$.
The definition of $N$ gives $f^{N+1}\zeta\notin U$.
Hence
$|f^{N+1}\zeta-f^{N+1}x|\geq C\sqrt{\tau}$
and we obtain a contradiction for sufficiently small $\tau$.
So the claim holds.

For $A,B>0$ we write $A\approx B$ if both $A/B$ and $B/A$ are bounded from above by constants independent of $\tau$, $\delta$, $b$.

\begin{lemma}\label{lema}
For every $k\leq N$, 
\begin{itemize}
\item[(a)] $D_k\approx \tau e^{-\lambda^u(\delta_Q)k}$;
\item[(b)] $\|w_k\|D_k\approx \tau$.
\end{itemize}
\end{lemma}
\begin{proof}
By the bounded distortion results in \cite[Section 6]{MorVia93} and \cite[Lemma 2.6(a)]{SenTak1}, 
$ \|w_{i}\|\approx e^{\lambda^u(\delta_Q)(i-1)}$ holds for every $1\leq i\leq k+1.$
Hence
$$D_k^{-1}= \frac{1}{\tau}\sum_{i=1}^k\frac{\|w_{i}\|^2}{\|w_{i+1}\|}\approx\frac{1}{\tau}\sum_{i=1}^k\|w_{k}\|e^{\lambda^u(\delta_Q)(i-k)}\approx \frac{1}{\tau} e^{\lambda^u(\delta_Q)k},$$
and so (a) holds.
(b) is contained in \cite[Lemma 2.4]{SenTak1}.
\end{proof}
If $1\leq i\leq p-1$,
then by \cite[Lemma 2.6]{SenTak1}, 
$\|D_zf ^{i}\left(\begin{smallmatrix}1\\0\end{smallmatrix}\right)\|\approx
\|w_{i+1}\|$ holds for all $z\in\gamma$, and $f^{i}\gamma$ is a $C^2(b)$-curve.
Lemma \ref{lema}(b) gives
\begin{equation*}
 {\rm length}(f^{i}\gamma)\approx{\rm length}(\gamma)\|w_{i+1}\|\leq
CD_{p-1}\|w_{i+1}\|\leq CD_{p-1}\|w_{p}\| \leq C\tau.\end{equation*}  
This
implies $x,fx,\ldots,f^{p}x\in U$, and so (a).

Split
$$D_{x}fv=A\cdot \left(\begin{smallmatrix}1\\0\end{smallmatrix}\right)+
B\cdot e^s(fx),\ \ A,B\in\mathbb R.$$
Since $v$ is in admissible position relative to $\zeta$, from the results in \cite{SenTak1,Tak11}
we have $A\approx |\zeta-x|$ and ${\rm length}(\gamma)\approx |\zeta-x|^2$.
On the other hand, \eqref{keq} and Lemma \ref{lema}(a) give
${\rm length}(\gamma)\approx D_p\approx \tau e^{-\lambda^u(\delta_Q)p}.$
Hence
\begin{equation*}
|\zeta-x|\approx  \frac{1}{\sqrt{\tau}}e^{-\frac{p}{2}\lambda^u(\delta_Q)}.
\end{equation*}
By Lemma \ref{lema}(b),
\begin{equation}\label{keq-2}
{\rm length}(f^{p-1}\gamma)\approx D_{p}\|w_{p}\| \approx \tau.\end{equation}  
Putting these estimates together we obtain
\begin{align*}
 |A|\cdot\|D_{fx}f^{p-1}\left(\begin{smallmatrix}1\\0\end{smallmatrix}\right)\|&\approx|\zeta-x
|\cdot\|w_p\|\approx|\zeta-x|\cdot\frac{{\rm length}(f^{p-1}\gamma)}{{\rm length}(\gamma)}\\
&\approx{\rm length}(f^{p-1}\gamma)\cdot |\zeta-x|^{-1}\approx \tau^{\frac{3}{2}}e^{\frac{p}{2}\lambda^u(\delta_Q)}.\end{align*}
For the other component in the splitting, (F3) gives 
$$|B|\cdot\|D_{fx}f^{p-1}e^s(fx)\|\leq (Cb)^{p-1}.$$
Then
$\|D_xf^pv\|\approx \tau^{\frac{3}{2}}e^{\frac{p}{2}\lambda^u(\delta_Q)}$, and so (b) holds.
It also follows that $s(D_xf^{p-1}v)\ll1$,
 and so $s(D_xf^{p}v)\leq\sqrt{b}$, and (c).
(d) follows from   \eqref{keq-2} and
$\mathcal F^s(f\zeta_0)\supset\alpha_0^+$,
which is a consequence of (F1) (F2) and (R4).
 \end{proof}

\begin{remark}
{\rm 
The existence of a uniform lower bound on $|f^px-f^p\zeta|$ can be read out from the above proof. However, this does not imply a
uniform lower bound on $|f^px-Q|$ as in 
Proposition \ref{binding}(d), because if $\zeta\neq\zeta_0$ then
$f^p\zeta$ escapes from $R$ to the left of $\alpha_0^-$.}
\end{remark}

 The integer $p$ and $\zeta$ in Proposition \ref{binding} are called a {\it bound period},
 and {\it a binding critical point} of $x$ respectively.
The next lemma allows us to find a binding critical point for any non wandering point which 
falls inside $I(\delta)$.
For $x\in \Omega$ let $e^u(x)$ denote any unit tangent vector which spans $E_x^u$.
\begin{lemma}\cite[Lemma 2.9]{SenTak1}\label{capture}
For any $x\in\Omega\cap I(\delta)\setminus\{\zeta_0\}$ there exists a critical point relative to which $e^u(x)$ is in admissible position. 
\end{lemma}

 \subsection{Unstable Lyapunov exponents of limit measures}\label{unslim} 
 The next proposition, which is a substantial improvement of  \cite[Proposition 4.3]{SenTak1},
  gives a lower estimate of the amount of drop of the unstable Lyapunov
 exponent in the weak limit of measures.
  Let $\mathcal M^e(f)$ denote the set of elements of $\mathcal M(f)$ which are ergodic.
  \begin{prop}\label{ly}
 Let $\{\mu_n\}_n$ be a sequence in $\mathcal M^e(f)$
 such that $\mu_n\to\mu$, $\mu=u\delta_Q+(1-u)\nu$, $0\leq u\leq 1$, $\nu\in\mathcal M(f)$
 and $\nu\{Q\}=0$. Then:
 \begin{equation}\label{lyeq1}
\frac{u}{2}\lambda^u(\delta_Q)+(1-u)\lambda^u(\nu)\leq  \liminf_{n\to\infty}\lambda^u(\mu_n);
 \end{equation}
 \begin{equation}\label{lyeq2}
\limsup_{n\to\infty}\lambda^u(\mu_n)\leq\lambda^u(\mu).
 \end{equation}
 \end{prop}
 \begin{proof}
  \eqref{lyeq2} was proved in the proof of \cite[Proposition 4.3]{SenTak1}.
Here we prove \eqref{lyeq1}.

 \begin{lemma}{\rm (\cite[Lemma 4.4]{SenTak1})}\label{44}
Let $\{\mu_n\}_n$ be a sequence in $\mathcal M^e(f)$ such that $\mu_n\to\mu$ and $\mu\{Q\}=0$.
Then $\lambda^u(\mu_n)\to\lambda^u(\mu)$.
\end{lemma} 
In the case $u=0$,  \eqref{lyeq1} is a consequence of Lemma \ref{44}.
We now consider the case $u\neq0$.
We begin by introducing a sequence $\{\tilde\alpha_k\}_{k=0}^\infty$ of compact curves in $W^s(P)\cap R$ 
 which allow us  to relate the proximity of an orbit's return to the boundary of the region $S$ with the time 
 it will subsequently spend near $Q$.
 Define $\tilde\alpha_0$ to be the connected component of $W^s(P)\cap R$ containing $P$,
 Given $\tilde\alpha_{k-1}$, define $\tilde\alpha_k$ to be one of the two components of $f^{-1}\tilde\alpha_{k-1}\cap R$ which is at the left of $\zeta_0$
 (See FIGURE 2). 

\begin{figure}
\begin{center}
\includegraphics[height=4cm,width=10cm]{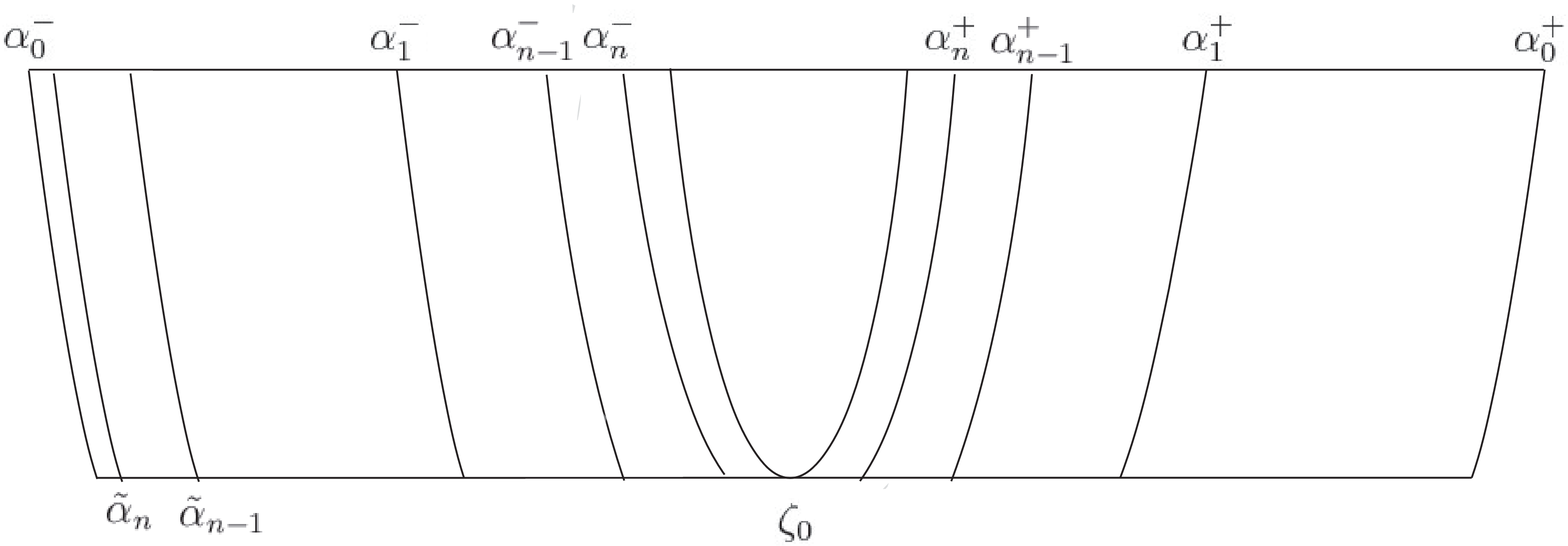}
\caption{The curves $\{\tilde\alpha_n\}$, $\{\alpha_n^{+}\}$, $\{\alpha_n^{-}\}$.
The $\{\tilde\alpha_n\}$ accumulate on the left stable side of $R$. Both $\{\alpha_n^{+}\}$ and $\{\alpha_n^{-}\}$ accumulate on the parabola $f^{-1}\alpha_0^+\cap R$
containing the point of tangency $\zeta_0$ near the origin.}
\end{center}
\end{figure}

Let $c\in(0,1/2)$ and define
$$X(c)=\frac{\lambda^u(\delta_Q)}{ 5}\left(\frac{1}{2}-c\right)\in(0,1).$$
Let $\tilde V_{k}$ denote the rectangle bordered by $\tilde\alpha_{k}$, $\alpha_0^-$ 
 and the unstable sides of $R$.
Define
$$V_{c,k}=\bigcup_{i=0}^{[1-X(c)]k}f^i\tilde V_{k},$$
where $[$ $\cdot$ $]$ denotes the integer part.
Observe that 
 $\{V_{c,k}\}_k$ is decreasing in $k$. 
By the Inclination Lemma, the Hausdorff distance between $\tilde\alpha_k$ and $\alpha_0^-$ converges to $0$ as $k\to\infty$.
This implies
  $\bigcap_{k=1}^\infty V_{c,k}=\alpha_0^-$.

\begin{lemma}\label{drop}
If $0<c_0<c<1/2$, then 
there exists 
$k_0\geq1$ such that if $k\geq k_0$ and $x\in\Omega$, $n\geq1$ are such that
$f^{-2}x\in I(\delta)$, $x\in \tilde V_{k}$, $x,fx,\ldots,f^{n-1}x\in V_{c,k}$ and $f^nx\notin V_{c,k}$, then
$$\|D_xf^n|E^u_x\|\geq e^{c_0\lambda^u(\delta_Q)n}.$$
\end{lemma}

\begin{proof}

Write $y$ for $f^{-2}x$. By Lemma \ref{capture} there exists a critical point
$\zeta$ relative to which $e^u(y)$ is in admissible position.
Let  $p=p(\zeta,y)$ denote the corresponding bound period.
We treat two cases separately.
\medskip

\noindent{\it Case 1: $p-2\leq n$.}
Proposition \ref{binding} gives
$\|D_{y}f^pe^u(y)\|\geq Ce^{\frac{\lambda^u(\delta_Q)}{2}p}$ and
$s(e^u(f^{p}y))\leq\sqrt{b}$. Since $f^{p}y,f^{p+1}y,\ldots,f^{n+1}y$ are located near $Q$, 
the bounded distortion gives
$\|D_{f^{p}y}f^{n+2-p}e^u(f^{p}y)\|\geq Ce^{\lambda^u(\delta_Q)(n+2-p)}.$
Then
\begin{align*}\|D_xf^ne^u(x)\|&=\frac{\|D_{y}f^{n+2}e^u(y)\|}
{\|D_{y}f^2e^u(y))\|}>\|D_{y}f^{n+2}e^u(y)\|\\
&=\|D_{f^{p}y}f^{n+2-p}e^u(f^{p}y)\|\cdot\|D_{y}f^pe^u(y)\|\\
&\geq C e^{\frac{\lambda^u(\delta_Q)}{2}(n+2)}\geq
e^{c_0\lambda^u(\delta_Q)n},\end{align*}
where the last inequality holds for sufficiently large $k$ because of 
$n> k$.

\medskip

\noindent{\it Case 2: $p-2> n$.}
Fix a $C^2(b)$-curve $\gamma$ which connects $f^{-1}x$ and $\alpha_0^+$.
The curves $f^i\gamma$ $(i=1,\ldots,n+1)$ are $C^2(b)$-curves
located near $Q$. 
The condition $f^{n-1}x\in V_{k}$ and $f^{n}x\notin V_{k}$ implies
\begin{equation}\label{Vkeq-1}
{\rm length}(f^{n+1}\gamma)>5^{-X(c)k}.\end{equation}

The bounded distortion gives
$\|D_zf ^{n+1}\left(\begin{smallmatrix}1\\0\end{smallmatrix}\right)\|\approx
\|D_{f\zeta}f ^{n+1}\left(\begin{smallmatrix}1\\0\end{smallmatrix}\right)\|$ for all $z\in\gamma$.
Using this and ${\rm length}(\gamma)\leq C|\zeta-y|^2$
we have
\begin{equation}\label{Vkeq-2}
{\rm length}(f^{n+1}\gamma)\approx{\rm length}(\gamma)\|D_{f\zeta}f ^{n+1}\left(\begin{smallmatrix}1\\0\end{smallmatrix}\right)\|\leq
C |\zeta-y|^2\|D_{f\zeta}f ^{n+1}\left(\begin{smallmatrix}1\\0\end{smallmatrix}\right)\|.\end{equation}  
Proposition \ref{lema}(a) gives
\begin{equation}\label{Vkeq-3}
|\zeta-y|\leq Ce^{-\frac{p}{2}\lambda^u(\delta_Q)}.
\end{equation}

 Split
$$D_{y}fe^u(y)=A\cdot \left(\begin{smallmatrix}1\\0\end{smallmatrix}\right)+
B\cdot e^s(fy),\ \
A, B\in\mathbb R.$$
Using \eqref{Vkeq-1} \eqref{Vkeq-2} \eqref{Vkeq-3} and $k<p$, for some $C>0$ we have
\begin{align*}
|A|\cdot\|D_{fy}f^{n+1}\left(\begin{smallmatrix}1\\0\end{smallmatrix}\right)\|&\approx|\zeta-y
|\cdot\|D_{f\zeta}f ^{n+1}\left(\begin{smallmatrix}1\\0\end{smallmatrix}\right)\|\\
&\geq C\cdot {\rm length}(f^{n+1}\gamma)|\zeta-y|^{-1}\geq C\cdot 5^{-X(c)p}e^{\frac{p}{2}\lambda^u(\delta_Q)}\\
&\geq C\cdot 2^{X(c)p}e^{c\lambda^u(\delta_Q)p}\geq
e^{c\lambda^u(\delta_Q)p},\end{align*}
where the last inequality holds provided $k$ is sufficiently large so that $C\cdot 2^{X(c)k}\geq1$.

For the other component in the splitting we have
$$|B|\cdot\|D_{fy}f^{n+1}e^s(fy)\|\leq(Cb)^{n+1}\leq (Cb)^p.$$
We have 
\begin{align*}\|D_xf^ne^u(x)\|&>\|D_{y}f^{n+2}e^u(y)\|\geq e^{c\lambda^u(\delta_Q)p}-(Cb)^p\geq  e^{
c_0\lambda^u(\delta_Q)p}>e^{c_0\lambda^u(\delta_Q)n}.
\end{align*}
where the second last inequality holds sufficiently large $k$
because of $c>c_0$ and $p>n+2>k$.
\end{proof}

Let us return to the proof of \eqref{lyeq1} in the case $u\neq0$.
Taking a subsequence if necessary, we may assume $\{\lambda^u(\mu_n)\}_n $
converges. 
Let
 $0<c_0<c<1/2$.
Fix a partition of unity $\{\rho_{0,c,k},\rho_{1,c,k}\}$ on $R$ such that
$${\rm supp}(\rho_{0,c,k})=\overline{\{x\in R\colon \rho_{0,c,k}(x)\neq0\}}\subset V_{c,k}\ \text{ and }
\ {\rm supp}(\rho_{1,c,k})\subset R\setminus V_{c,2k}.$$
\begin{claim}\label{claim3}
$\displaystyle{\liminf_{n\to\infty}}\mu_n(V_{c,k})\geq u$.
\end{claim}
\begin{proof}
If $u\neq1$, then $\mu_n-u\delta_Q\to(1-u)\nu$.
Since $\nu\{Q\}=0$, $\nu(\partial V_{c,k})=0$. This yields
$(\mu_n-u\delta_Q)(V_{c,k})\to(1-u)\nu(V_{c,k})$ as $n\to\infty$,
namely $\mu_n(V_{c,k})\to u+(1-u)\nu(V_{c,k})$.
The same convergence obviously takes place in the case $u=1$.
Hence the claim holds.
\end{proof}

From the Ergodic Theorem, there exists $\xi_{n}\in\Omega$ such that
$$\lim_{m\to\infty}\frac{1}{m}\#
\{0\leq i\leq m-1\colon f^i\xi_n\in V_{c,k}\}=\mu_n(V_{c,k}).$$ The forward orbit of $\xi_n$ is a concatenation of segments in $V_{c,k}$ and those out of $V_{c,k}$.
Lemma \ref{drop} gives
$$\int\rho_{0,c,k}\log J^ud\mu_n=
\lim_{m\to\infty}\frac{1}{m}\sum_{i=0}^{m-1} (\rho_{0,c,k}\log
J^u)\circ f^i(\xi_n)\geq \mu_n(V_{c,k})c_0\lambda^u(\delta_Q).$$
If $u\neq 1$, then the weak convergence for the sequence $\{(1-u)^{-1}(\mu_n-u\delta_Q)\}_n$ in $\mathcal M(f)$ implies
$$\lim_{n\to\infty}\int\rho_{1,c,k}\log J^ud\mu_n=(1-u)\int\rho_{1,c,k}\log J^ud\nu.$$
The same inequality remains to hold for the case $u=1$.
Hence we have
\begin{align*}\lim_{n\to\infty}\lambda^u(\mu_n)&=
\lim_{n\to\infty}\int\rho_{0,c,k}\log
J^ud\mu_n+\lim_{n\to\infty}\int\rho_{1,c,k}\log J^ud\mu_n\\
& \geq uc_0\lambda^u(\delta_Q)+(1-u)\int\rho_{1,c,k}\log J^ud\nu.
\end{align*}
Since $\nu\{Q\}=0$, $\rho_{1,c,k}\log J^u\to \log J^u$ $\nu$-a.e. as $k\to\infty$. Letting $k\to\infty$ and then using the Dominated Convergence Theorem gives
\begin{align*}\lim_{n\to\infty}\lambda^u(\mu_n) \geq uc_0\lambda^u(\delta_Q)+(1-u)\lambda^u(\nu).
\end{align*}
Since $c,c_0$ are arbitrary such that $0<c_0<c<1/2$,
the desired inequality holds.
\end{proof}

\subsection{Construction of measures with small unstable Lyapunov exponents}\label{lyapatom}

In this subsection we construct a sequence of atomic measures with small unstable Lyapunov exponent. In this and the next subsections we assume $f$ preserves orientation.

We continue using the sequence $\{\tilde\alpha_n\}_{n=0}^\infty$ of compact 
curves in $W^s(P)\cap R$ in the proof of Proposition \ref{ly}. 
 For each $n\geq0$ let $\alpha_n'$ denote the connected 
component of $f^{-1}\tilde\alpha_n\cap R$ 
which is not $\tilde\alpha_{n+1}$. 
The set $f^{-1}\alpha_n'\cap R$ consists of two curves, one at the left
of $\zeta_0$ and the other at the right. They are denoted by $\alpha_{n+1}^-$,
$\alpha_{n+1}^+$ respectively (See FIGURE 2).
By definition, these curves obey the following diagram
\begin{equation}\label{diagram}
\{\alpha_{n+1}^-,\alpha_{n+1}^+\}\stackrel{f^2}{\to}\tilde\alpha_n
\stackrel{f}{\to}\tilde\alpha_{n-1}
\stackrel{f}{\to}\tilde\alpha_{n-2}\stackrel{f}{\to}\cdots
\stackrel{f}{\to}\tilde\alpha_1=\alpha_1^-\stackrel{f}{\to}
\tilde\alpha_0=\alpha_1^+.\end{equation}
Observe that $\tilde\alpha_0=\alpha_1^+$ and $\tilde\alpha_1=\alpha_1^-$.

Let $\omega^+_n$ (resp. $\omega^-_n$) denote the rectangle bordered by $\alpha_n^+$, $\alpha_{n+1}^+$ (resp.
$\alpha_n^-$, $\alpha_{n+1}^-$)
and the unstable sides of $R$. 
The following holds:

\begin{itemize}

\item $f\omega_n^\pm$ is at the right of $\alpha_1^+$. If $n\geq 2$, then
$f^i\omega_n^\pm$ is at the left of $\alpha_1^-$ for every $2\leq i\leq n$;

\item One of the stable sides
of $f^{n+1}\omega_n^\pm$ 
 is contained in $\alpha_1^-$ and the other in $\alpha_1^+$.
The unstable sides of $f^{n+1}\omega_n^\pm$ 
are $C^2(b)$-curves connecting $\alpha_1^-$ and $\alpha_1^+$ \cite{SenTak2}.

\end{itemize}

For each $n\geq4$
define $$\hat\omega_n=\omega_{n-3}^+\cap f^{-n+2}\omega_1^-.$$
Since $f$ preserves orientation, $f^{-2}\zeta_0$ is contained in the unstable sides of $\omega_1^-$ (See FIGURE 3).

\begin{figure}
\begin{center}
\includegraphics[height=5cm,width=11cm]{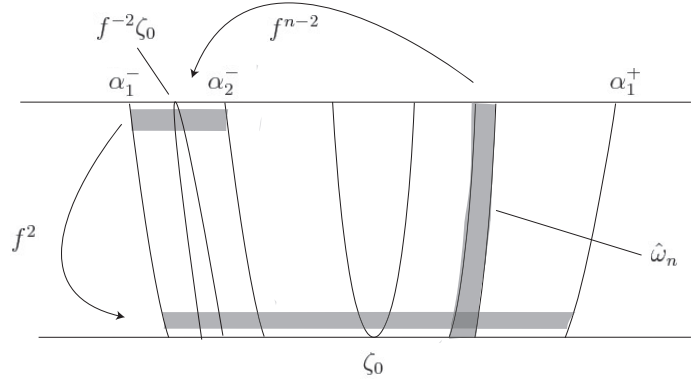}
\caption{The rectangles $\hat\omega_n$, $f^{n-2}\hat\omega_n$, $f^{n}\hat\omega_n$.}
\end{center}
\end{figure}

Let $\gamma^u(\zeta_0)$ denote the $C^2(b)$-curve in $W^u$ which contains $\zeta_0$, and connects $\alpha_1^-$ and $\alpha_1^+$.
Define $$A_n=
 \{x\in\hat\omega_n\cap\Omega\colon \min\{|x-y|\colon y\in\gamma^u(\zeta_0)\}\leq (Cb)^{\frac{n}{2}}\}.$$
 If $x\in A_n$ then $|x-\zeta_0|> 5^{-n}$, for otherwise $f^{n}x$ were
close to $f^{n}\zeta_0\in\alpha_0^-$, a contradiction.

\begin{prop}\label{estimate}
Assume $f$ preserves orientation. For sufficiently large $n$ and all $x\in A_n$,
$$\|D_{x}f^{n}|E_{x}^u\|\leq Ce^{\frac{n}{2}\lambda^u(\delta_Q)}.$$
\end{prop}

\begin{proof}
Let $n$ be large enough so that $A_n\subset I(\delta)$.
Let $x\in A_n$.
We show that 
$e^u(x)$ is in admissible position relative to $\zeta_0$.
Since $\gamma^u(\zeta_0)$ is $C^2(b)$ and 
 $|x-\zeta_0|> 5^{-n}$, this would hold if
  $|x-z|\leq Cb^{\frac{n}{2}}$ and 
$\angle(E_x^u,T_z\gamma^u(\zeta_0))\leq Cb^{\frac{n}{8}}$,
where $z$ denote the point on $\gamma^u(\zeta_0)$ whose first coordinate coincides with that of $x$.
The first inequality immediately follows from the definition of $A_n$.
The second one follows from the sublemma below, combined with the following fact
from \cite[Sect.2.4 $\&$ Lemma 2.8]{SenTak1}: 
there are a sequence $\{x_n\}_n$ in $\Omega$ and a sequence $\{\gamma_n\}_n$ of $C^2(b)$-curves in $W^u$
connecting $\alpha_1^-$ and $\alpha_1^+$ 
such that $x_n\in\gamma_n$ and
$T_{x_n}\gamma_n\to E_x^u$ as $n\to\infty$.

\begin{sublemma}\label{angleint}
Let $L\in(0,1/4)$, and let
$\gamma_i=\{(x,\gamma_i(x))\in\mathbb R^2\colon x\in [-L,L]\}$, 
$i=1,2$ be 
two disjoint $C^2(b)$-curves.
Assume $|\gamma_1(0)-\gamma_2(0)|\leq L^2$.
Then $|\gamma_1'(0)-\gamma_2'(0)|\leq\sqrt{L}.$
\end{sublemma}
\begin{proof}
Without loss of generality we may assume $\gamma_2$ lies above $\gamma_1$, 
and $\gamma_1'(0)>\gamma_2'(0)$.
Set $A(x)=\gamma_2(x)-\gamma_1(x)$. By the assumption, $A(x)>0$ for all $x\in [-L,L]$ 
 and $A'(0)<0$.
By the $C^2(b)$-property, $|A''(x)|\leq 2\sqrt{b}$.
If $A'(0)<-\sqrt{L}$, then by the Mean Value Theorem,
$A'(x)\leq A'(0)+2\sqrt{b}x\leq-\sqrt{L}/2,$
and thus
$A(L)=A(0)+\int_{0}^LA'(x)dx\leq L^2-L^{3/2}/2<0,$
a contradiction.
\end{proof}

Let $p=p(\zeta_0,x)$ denote the bound period.
Proposition \ref{binding}(d) implies
$n-p\leq C$, and therefore
$$\|D_{x}f^{n}|E_{x}^u\|=\|D_{x}f^{p}|E_{x}^u\|\cdot\|D_{f^{p}x}f^{n-p}|E_{f^{p}x}^u\|
\leq Ce^{\frac{p}{2}\lambda^u(\delta_Q)}\cdot 5^{n-p}\leq Ce^{\frac{n}{2}\lambda^u(\delta_Q)}.\qedhere$$
\end{proof}




\begin{cor}\label{estcor}
If $f$ preserves orientation,
then $(1/2)\lambda^u(\delta_Q)\geq\lambda_m^u.$
\end{cor}
\begin{proof}
By the result in \cite{Tak14}, 
the subset $\bigcap_{k=-\infty}^\infty (f^n)^k\hat\omega_n$ of $A_n$ is a singleton which 
consists of a hyperbolic periodic point of period $n$.
Proposition \ref{estimate} gives an upper estimate of the unstable Lyapunov exponent of
the atomic probability measure on the orbit of this periodic point.
Since $n$ is arbitrary, this
yields the desired inequality. \end{proof}





\subsection{Lower estimate of the pressure}\label{abs}
The next lemma will be used to derive a contradiction in the proof of Theorem A.

\begin{lemma}\label{absence}
Assume $f$ preserves orientation.
If $(1/2)\lambda^u(\delta_Q)= \lambda_m^u$, then $P(t)>-t\lambda_m^u$
for any $t>0$.
\end{lemma}
\begin{proof}
We adapt the construction of Leplaideur \cite{Lep11} which was inspired by Makarov and Smirnov \cite{MakSmi03}.
The idea is to construct a uniformly hyperbolic subset which supports an invariant measure whose minus of the free energy 
is slightly bigger than $-t\lambda_m^u$. 

Let $q>0$ be a square of a large integer. We use the rectangles $\hat \omega_n$ $(n=q-\sqrt{q}+1,\ldots,q)$
to construct an induced system. Set $r_n=q-\sqrt{q}+n$ $(n=1,\ldots,\sqrt{q})$.
Endow $\Sigma_{\sqrt{q}}=\{\underline{a}=\{a_i\}_{i\in\mathbb Z}\colon a_i\in\{1,\ldots,\sqrt{q}\}\}$ with the product topology 
of the discrete topology.
Define $\pi\colon\Sigma_{\sqrt{q}}\to\Omega$ by
$\pi(\underline{a})=x$, where
 $$\omega^s_k=\hat\omega_{a_0}\cap\left(\bigcap_{i=1}^k 
f^{-r_{a_0}}\circ\cdots\circ f^{-r_{a_{i-1}}}\hat\omega_{a_i}\right)\ \text{ and }\
\omega^u_k=\bigcap_{i=1}^{k}f^{r_{a_{-1}}}\circ\cdots\circ f^{r_{a_{-i}}}\hat\omega_{a_{-i}},$$
and
$$\{x\}=\left(\bigcap_{k=1}^{\infty}\omega_k^s\right)
\cap\left(\bigcap_{k=1}^{\infty}\omega_k^u\right).$$ 
By \cite{Tak14},
$\pi$ is well-defined, continuous, injective.

Let $\sigma\colon\Sigma_{\sqrt{q}}\circlearrowleft$ denote the left shift.
For a $\sigma^q$-invariant Borel probability measure $\mu$,
define a measure $\mathcal L(\mu)$ by
$$\mathcal L(\mu)=\sum_{[a_0,a_1,\ldots,a_{q-1}]}\sum_{i=0}^{r_{a_0}+r_{a_1}+\cdots+r_{a_{q-1}}-1} f_*^i(\pi_*(\mu |_{[a_0,a_1,\ldots,a_{q-1}]})),$$
where $[a_0,a_1,\ldots,a_{q-1}]=\{\underline{b}\in\Sigma_{\sqrt{q}}\colon b_i=a_i\ \ i=0,1,\ldots,q-1\}.$
Then $\mathcal L(\mu)$ is a probability and $f$-invariant. 
Define $r\colon \Sigma_{\sqrt{q}}\to\mathbb R$ by
$r(\underline{a})=\sum_{i=0}^{q-1}r_{a_i}$, and
 $\Phi_t\colon\Sigma_{\sqrt{q}}\to\mathbb R$ by 
$$\Phi_t(\underline{a})=\sum_{i=0}^{r(\underline{a})-1}\varphi_t(f^i(\pi(\underline{a}))).$$
Set $P_n=\{\underline{a}\in\Sigma_{\sqrt{q}}\colon \sigma^{qn}(\underline{a})=\underline{a}\}.$ 
By Proposition \ref{estimate} and the assumption $(1/2)\lambda^u(\delta_Q)= \lambda_m^u$,
for each $\underline{a}\in P_n$ we have
$$\frac{\Phi_t(\underline{a})}{r(\underline{a})}\geq 
-\frac{r_{a_0}Ct+(r_{a_1}+\cdots+r_{a_{q-1}}-1)t\lambda_m^u}{r_{a_1}+\cdots+r_{a_{q-1}}-1}
\geq -\frac{Ct}{q} -t\lambda_m^u.$$
Let $\mu_0$ denote the measure of maximal entropy of $\sigma^q$.
Since $r$ and $\Phi_t$ are continuous,
as $n\to\infty$ we have
$$\frac{1}{\#P_n}\sum_{\underline{a}\in P_n} r(\underline{a})\to \int r d\mu_0\ \text{ and }\
 \frac{1}{\#P_n}\sum_{\underline{a}\in P_n} \Phi_t(\underline{a})\to\int \Phi_t d\mu_0.$$
 Hence
$$\frac{\int\Phi_td\mu_0}{\int r d\mu_0} \geq -\frac{Ct}{q} -t\lambda_m^u.$$
Since the entropy of $\mu_0$ is $q\log\sqrt{q}$ and $\int r d\mu_0\leq q^2$,
we obtain
\begin{align*}
h(\mathcal L(\mu_0))-t\int \log J^u d\mathcal L(\mu_0)&=\frac{1}{\int r d\mu_0}
\left(q\log\sqrt{q}+\int \Phi_td\mu_0\right)\\
&\geq \frac{1}{q}\log\sqrt{q}-t\frac{C}{q}-t\lambda_m^u>-t\lambda_m^u.
\end{align*}
The last inequality holds for sufficiently large $q$.
\end{proof}

\section{Proofs of the theorems}
In this section we put together the results in Sect.2 and prove the theorems.

\subsection{Existence of equilibrium measures}\label{equilibrium}
We prove Theorem A.

\medskip

\noindent{\it Proof of Theorem A.} Let $t>0$.
Corollary \ref{estcor} gives
\begin{equation}\label{equi2}P(t)\geq -t\lambda^u_m\geq-(t/2)\lambda^u(\delta_Q).\end{equation}
By the linearity of entropy and unstable Lyapunov exponent on measures,
one can choose a sequence $\{\mu_n\}_n$ in $\mathcal M^e(f)$ such that $F_{\varphi_t}(\mu_n)\to P(t)$. Choosing a subsequence if necessary we may assume 
$\mu_n\to\mu\in\mathcal M(f)$. 
Write $\mu= u\delta_Q+(1-u)\nu$, $0\leq u\leq1$, $\nu\{Q\}=0$.
From the upper semi-continuity of entropy \cite[Corollary 3.2]{SenTak1} and \eqref{lyeq1},
\begin{align}\label{equi1}
P(t)=\lim_{n\to\infty}F_{\varphi_t}(\mu_n)&\leq \limsup_{n\to\infty}h(\mu_n)-t\liminf_{n\to\infty}\lambda^u(\mu_n)\\
&\leq h(\mu)-t\left(\frac{u}{2}\lambda^u(\delta_Q)+(1-u)\lambda^u(\nu)\right)\notag\\
&=(1-u)F_{\varphi_t}(\nu)-\frac{tu}{2}\lambda^u(\delta_Q).\notag\end{align}
For the last equality we have used
$h(\mu)=(1-u)h(\nu)$. Plugging
$-(tu/2)\lambda^u(\delta_Q)\leq uP(t)$ from \eqref{equi2} into \eqref{equi1} we obtain
\begin{equation}\label{equi3}P(t)\leq(1-u)F_{\varphi_t}(\nu)+uP(t).\end{equation}
If $u\neq1$, then rearranging \eqref{equi3} yields $P(t)\leq F_{\varphi_t}(\nu)$. Namely $\nu$ is an equilibrium measure
for $\varphi_t$.
If $u=1$, then 
  \eqref{equi1} and Corollary \ref{estcor} yield $P(t)\leq -(t/2)\lambda^u(\delta_Q)\leq -t\lambda_m^u$.
From this and \eqref{equi2} we obtain
$P(t)= -t\lambda^u_m=-(t/2)\lambda^u(\delta_Q)$, a contradiction to Lemma \ref{absence}.
\qed

\subsection{Accumulation points of equilibrium measures as $t\to+\infty$}\label{zerot}
We now prove Theorem B.
\medskip

\noindent{\it Proof of Theorem B.}
Let $\{\mu_{t_n}\}_{n=0}^\infty$ be a sequence in $\mathcal M(f)$ such that 
$t_n\nearrow+\infty$, 
$\mu_{t_n}$ is an ergodic equilibrium measure for $\varphi_{t_n}$
and $\mu_{t_n}\to\mu$ as $n\to\infty$.
Write
$\mu= u\delta_Q+(1-u)\nu$, $0\leq u\leq1$, $\nu\{Q\}=0$.
Proposition \ref{ly} gives
\begin{equation}\label{eqzero}
\lambda_m^u=\lim_{n\to\infty}\lambda^u(\mu_{t_n})\geq\frac{u}{2}\lambda^u(\delta_Q)+(1-u)\lambda^u(\nu).\end{equation}

By Corollary \ref{estcor}, $(1/2)\lambda^u(\delta_Q)\geq \lambda_m^u$.
In this paragraph we treat the case $(1/2)\lambda^u(\delta_Q)>\lambda_m^u$.
If $u=1$, then \eqref{eqzero} gives $\lambda_m^u\geq (1/2)\lambda^u(\delta_Q)$, a contradiction.
Hence $u\neq1$.
Plugging $(u/2)\lambda^u(\delta_Q)\geq u\lambda_m^u$ into \eqref{eqzero} and then rearranging the result yield
$(1-u)\lambda_m^u\geq (1-u)\lambda^u(\nu)$, and thus
 $\lambda_m^u\geq\lambda^u(\nu)$.
 If $u\neq0$, all these three inequalities are strict and we reach
  a contradiction. 
Hence $u=0$ and so $\mu=\nu$.
Lemma \ref{44} gives $\lambda^u(\mu)=\displaystyle{\lim_{n\to\infty}}\lambda^u(\mu_{t_n})=\lambda_m^u$.

In the case $(1/2)\lambda^u(\delta_Q)=\lambda_m^u$,
\eqref{eqzero} gives $(1-u)\lambda_m^u\geq(1-u)\lambda^u(\nu)$.
If $u\neq1$, then $\nu$ is Lyapunov minimizing.
If $u=1$, then $\mu=\delta_Q$. \qed

\subsection{Entropy criterion}\label{criterion}
We now prove Theorem C.
\medskip

\noindent{\it Proof of Theorem C.}
Let $\{\mu_{t_n}\}_{n=0}^\infty$ and $\mu$ be the same as in the proof of Theorem B.
Assume
$(1/2)\lambda^u(\delta_Q)\neq\lambda_m^u$.
\begin{lemma}\label{claim2}
For any $\varepsilon>0$ there exists $N>0$ such that for all $n\geq N$,
$$P(t_n)<h(\mu)-t_n\lambda^u(\mu)+\varepsilon.$$
\end{lemma}
\begin{proof}
Suppose the statement is false.
Then, there exists $\varepsilon>0$ such that
$P(t_n)\geq h(\mu)-t_n\lambda^u(\mu)+\varepsilon$
holds for infinitely many $n$.
For these $n$ we have 
$$P(t_n)=h(\mu_{t_n})-t_n\lambda^u(\mu_{t_n})\geq h(\mu)-t_n\lambda^u(\mu)+\varepsilon.$$
Since $\mu$ is Lyapunov minimizing by Theorem B,
 $\lambda^u(\mu)\leq\lambda^u(\mu_{t_n})$. Using this and
 the upper semi-continuity of entropy, for sufficiently large $n$ we have
$$h(\mu)+\frac{\varepsilon}{2}-t_n\lambda^u(\mu)>h(\mu_{t_n})-t_n\lambda^u(\mu_{t_n}).$$
These two inequalities yield a contradiction.
\end{proof}

Suppose there exists a Lyapunov minimizing measure $\nu$ such that $h(\mu)<h(\nu)$.
Let $\varepsilon>0$ be such that $h(\mu)+\varepsilon<h(\nu)$.
By Lemma \ref{claim2}, for sufficiently large $n$ we have 
 $$P(t_n)<h(\mu)-t_n\lambda^u(\mu)+\varepsilon<h(\nu)-t_n\lambda^u(\mu)=h(\nu)-t_n\lambda^u(\nu),$$
 a contradiction. 
\qed

\subsection{Convergence point of equilibrium measures as $t\to-\infty$}
We now prove Theorem D.
\medskip

\noindent{\it Proof of Theorem D.}
The lemma below shows that $\delta_Q$ is the unique measure which maximizes the unstable Lyapunov exponent.
\begin{lemma}\label{maximize}
For any $\mu\in\mathcal M(f)\setminus\{\delta_Q\}$, $\lambda^u(\mu)<\lambda^u(\delta_Q)$.
\end{lemma}

\begin{proof}
From the ergodic decomposition theorem, we only have to consider ergodic measures.

One can choose a neighborhood $W$ of $Q$ such that for any ergodic 
$\mu\in\mathcal M(f)\setminus\{\delta_Q\}$ with ${\rm supp}(\mu)\cap I(\delta)=\emptyset$,
${\rm supp}(\mu)\cap W=\emptyset$ holds.
Hence, for such $\mu$, $\lambda^u(\mu)<\lambda^u(\delta_Q)$ holds.

It is left to consider the case ${\rm supp(\mu)}\cap I(\delta)\neq\emptyset$.
Then $\mu(I(\delta))>0$, and from the Ergodic Theorem, it is possible to
take a point $x\in\Omega$ such that $$\displaystyle{\lim_{n\to\infty}}\frac{1}{n}\displaystyle{\sum_{i=0}^{n-1}}
\log \|D_{f^ix}f|E^u_{f^ix}\|=\lambda^u(\mu),$$
and
$$\lim_{n\to\infty}\frac{1}{n}\#\{0\leq i\leq n-1\colon f^ix\in I(\delta)\}=\mu(I(\delta)).$$
Let $0\leq n_1<n_2<\cdots$ be the sequence of integers such that
$f^{n_k}x\in I(\delta)$ and $f^{n_k+i}x\notin I(\delta)$ for every $1\leq i\leq n_{k+1}-n_k-1$.
Let $p_k$ denote the bound period for $f^{n_k}x$.
Then
$$\|D_{f^{n_k}x}f^{p_k}|E^u_{f^{n_k}x}\|\leq Ce^{\frac{\lambda^u(\delta_Q)}{2}p_k}.$$
For iterates out of $I(\delta)$,
$$\|D_{f^{n_k+p_k}x}f^{n_{k+1}-n_k-p_k}|E^u_{f^{n_k+p_k}x}\|\leq Ce^{\lambda^u(\delta_Q)(n_{k+1}-n_k-p_k)}.$$
Since $p_k\geq -C\log\delta$ by \cite{SenTak1},
$Ce^{-\frac{\lambda^u(\delta_Q)}{2}p_k}\leq \delta^C$. Hence
$$\|D_{f^{n_k}x}f^{n_{k+1}-n_k}|E^u_{f^{n_k}x}\|\leq Ce^{-\frac{\lambda^u(\delta_Q)}{2}p_k}
\cdot e^{\lambda^u(\delta_Q)(n_{k+1}-n_k)}\leq \delta^Ce^{\lambda^u(\delta_Q)(n_{k+1}-n_k)}.$$
This yields
\begin{align*}\lim_{k\to\infty }\frac{1}{n_k}\sum_{i=0}^{n_k-1}\log \|D_{f^ix}f|E^u_{f^ix}\|&\leq \lambda^u(\delta_Q)+C\log\delta\cdot
\lim_{k\to\infty}\frac{1}{n_k}\#\{0\leq i\leq n_k-1\colon f^ix\in I(\delta)\}\\
&= \lambda^u(\delta_Q)+C\log\delta\cdot\mu(I(\delta))<\lambda^u(\delta_Q).\qedhere\end{align*}
\end{proof}

Let $\{\mu_{t_n}\}_{n=0}^\infty$ be a sequence in $\mathcal M(f)$ such that 
$t_n\searrow-\infty$, 
$\mu_{t_n}$ is an ergodic equilibrium measure for $\varphi_{t_n}$
and $\mu_{t_n}\to\mu$ as $n\to\infty$.
We have $\sup\{\lambda^u(\mu)\colon\mu\in\mathcal M(f)\}=\displaystyle{\limsup_{n\to\infty}}\lambda^u(\mu_{t_n})$,
and the upper semi-continuity of the unstable Lyapunov exponent in \eqref{lyeq2}
gives $\displaystyle{\limsup_{n\to\infty}}\lambda^u(\mu_{t_n})\leq\lambda^u(\mu)$.
By Lemma \ref{maximize}, $\mu=\delta_Q$.
Hence Theorem D holds.
\qed

\subsection*{Acknowledgments}
Partially supported by the Grant-in-Aid for Young Scientists (B) of the JSPS, Grant No.23740121.

\bibliographystyle{amsplain}

\end{document}